\documentclass[a4paper,12pt]{amsart}
\usepackage{amscd}
\usepackage{amsfonts}
\usepackage{amsmath}
\usepackage{amssymb}
\usepackage{latexsym}
\usepackage{amsthm}
\usepackage{color}
\usepackage[all]{xy}

\usepackage[dvipdfmx]{graphicx}
\usepackage{latexsym}
\usepackage{color}
\newtheorem{proposition}{Proposition}
\newtheorem{lemma}{Lemma}[section]
\newtheorem{theorem}{Theorem}
\newtheorem{corollary}{Corollary}

\theoremstyle{definition}

\theoremstyle{definition}

\theoremstyle{definition}

\newcommand{\R}{\mathbb{R}}

\begin{document}

\title[Stability of $C^\infty$ convex integrands]
{Stability of $C^\infty$ convex integrands}
\thanks{\color{black} This work is partially supported by JSPS and CAPES under the Japan--Brazil research cooperative program and 
JSPS KAKENHI Grant Number 26610035.}
%%%%%%%%%%%%%%%%%%%%%%%%%%%%%%%%%%%%%%%%%%%%%%  
\author[E.B.~Batista]{Erica Boizan Batista}
\address{
%Research Group of Mathematical Sciences,  
Research Institute of Environment and Information Sciences,  
Yokohama National University, 
Yokohama 240-8501, Japan}
\email{ericabbatista@gmail.com}
\author[H.~Han]{Huhe Han}
\address{Graduate School of Environment and Information Sciences,Yokohama National University, {Yokohama 240-8501,} Japan}
\email{han-huhe-bx@ynu.jp}
\author[T.~Nishimura]{Takashi Nishimura%\thanks{Corresponding author.} 
}
\address{
%Research Group of Mathematical Sciences,  
Research Institute of Environment and Information Sciences,  
Yokohama National University, 
Yokohama 240-8501, Japan}
\email{nishimura-takashi-yx@ynu.jp}
%%%%%%%%%%%%%%%%%%%%%%%%%%%%%%%%%%%%%%%%%%%%%%%  

%\author[E.B. Batista, H. Han and  T. Nishimura]{Erica Boizan Batista, 
%Huhe Han and Takashi Nishimura}

\begin{abstract}
In this paper, it is shown that 
the set consisting of stable convex integrands $S^n\to \mathbb{R}_+$ is 
open and dense in the set consisting of $C^\infty$ convex integrands 
with respect to Whitney $C^\infty$ topology.    
Moreover, an application of the proof of this result is also shown.   
 \end{abstract}

%{\bf Mathematics Subject Classification (2010):} 58K40 (primary), 57R45, 
%58K20 (secondary). \\

\date{}

\maketitle

\section{Introduction}
In the celebrated series 
\cite{mather1, mather2, mather3, mather4, mather5, mather6}, 
J.~Mather gave a complete answer to the problem on density 
of proper stable mappings in a surprizing form.      
For proper $C^\infty$ mappings of special type, 
it is natural to ask the similar question, namely to ask  
\lq\lq Are generic proper mappings of special type stable ?\:\rq\rq.              
%\; 
Such investigations, for instance, can be found in \cite{generic projections} 
for generic projections of submanifolds, 
%in \cite{arnold-guseinzade-varchenko} 
%for Lagrangian stability and Legendrian stablity,  
in \cite{bruce-kirk} for generic projections of stable mappings 
and  in \cite{ichiki-nishimura1, 
ichiki-nishimura2, ichiki-nishimura3, ichiki-nishimura-oset-ruas} 
for generic distance-squared mappings and their generalizations.      
\par 
Motivated by these researches, in this paper, it is investigated 
the density problem for 
$C^\infty$ convex integrands.   The notion of convex integrand 
was firstly introduced in \cite{taylor}, which is defined as follows.    
For a positive integer $n$, let $S^n$ be the unit sphere of 
$\mathbb{R}^{n+1}$.   The set consisting of positive real numbers is denoted by $\mathbb{R}_+$.     Then, a continuous function 
$\gamma: S^n\to \mathbb{R}_+$  is called a {\it convex integrand} if 
the boundary of the convex hull of 
 inv(graph$(\gamma)$) is exactly the same set as inv(graph$(\gamma)$),  
where graph$(\gamma)$ is the set 
$\left\{(\theta, \gamma(\theta))\; \left|\; \theta\in S^n\right.\right\}$ 
with respect to the polar plot expression for $\mathbb{R}^{n+1}-\{0\}$ 
and inv$: \mathbb{R}^{n+1}-\{0\}\to \mathbb{R}^{n+1}-\{0\}$ is the inversion 
defined by inv$(\theta, r)=\left(-\theta, \frac{1}{r}\right)$.     
The notion of convex integrand is closely related with the notion of 
Wulff shape, which was firstly introduced in \cite{wulff} as a geometric model of crystal at equilibrium.      Integration of a convex integrand $\gamma$ 
over $S^n$ represents the surface energy of the Wulff shape associated 
with $\gamma$.    Hence, $\gamma$ is called a convex {\it integrand}.   
For more details on convex integrands, see for instance 
\cite{morgan, taylor}.   
\par 
Set 
\[
C^\infty_{\rm conv}(S^n, \mathbb{R}_+)
=\left\{\gamma\in C^\infty(S^n, \mathbb{R}_+)\; |\; 
\gamma \mbox{ is a convex integrand}\right\}, 
%\mbox{th
%e boundary of the convex hull of the image of
%inv(graph}%(\gamma)=\mbox{the image of inv(graph}(\gamma)\right\}, 
\]   
where $C^\infty(S^n, \mathbb{R}_+)$ is the set consisting of 
$C^\infty$ functions $S^n\to \mathbb{R}_+$.    
The set $C^\infty(S^n, \mathbb{R}_+)$ is endowed 
with {\it Whitney $C^\infty$ topology} 
(for details on Whitney $C^\infty$ topology, 
see for instance \cite{arnold-guseinzade-varchenko, g-g}); 
and the set $C^\infty_{\rm conv}(S^n, \mathbb{R}_+)$ 
is a topological subspace of $C^\infty(S^n, \mathbb{R}_+)$.    
Two $C^\infty$ functions $\gamma_1, \gamma_2: S^n\to \mathbb{R}_+$ are 
said to be $\mathcal{A}$-{\it equivalent} if there exist $C^\infty$ 
diffeomorphisms $h: S^n\to S^n$ and $H: \mathbb{R}_+\to \mathbb{R}_+$ 
such that the equality $\gamma_2=H\circ \gamma_1\circ h^{-1}$ holds.    
A $C^\infty$ function $\gamma\in C^\infty(S^n, \mathbb{R}_+)$ is 
said to be {\it stable} 
if the $\mathcal{A}$-equivalence class $\mathcal{A}(\gamma)$ is 
{\color{black}an} open subset 
of the topological space $C^\infty(S^n, \mathbb{R}_+)$.     
By definition, any function $\mathcal{A}$-equivalent to a stable function 
is stable.    
Set  
\[
S^\infty(S^n, \mathbb{R}_+)= 
\left\{\gamma\in C^\infty(S^n, \mathbb{R}_+)\; |\; \gamma  
\mbox{ is stable}\right\}.    
\] 
By definition, $S^\infty(S^n, \mathbb{R}_+)$ is open.    
The following proposition is one of corollaries of Mather's series \cite{mather1, mather2, mather3, mather4, mather5, mather6}.  
% is the following:      
\begin{proposition}\label{proposition 1}
\begin{enumerate}
\item[(1)]\quad A $C^\infty$ function $\gamma\in C^\infty(S^n, \mathbb{R}_+)$ is stable if and only if 
 all critical points of $\gamma$ are non-degenerate and  
$\gamma(\theta_1)\ne \gamma(\theta_2)$ 
holds 
%is satisfied 
for any two distinct critical points 
$\theta_1, \theta_2\in S^n$.  
\item[(2)]\quad The open subset $S^\infty(S^n, \mathbb{R}_+)$ is dense 
in $C^\infty(S^n, \mathbb{R}_+)$. 
\end{enumerate}  
\end{proposition}
\noindent  
The assertion (2) of Proposition \ref{proposition 1} asserts that any $C^\infty$ function 
$\gamma: S^n\to \mathbb{R}_+$ can be perturbed to a stable function 
$\widetilde{\gamma}$ by a sufficiently small perturbation;  
and for any sufficiently small $\varepsilon >0$, any continuous mapping 
$\Phi: (-\varepsilon, \varepsilon)\to C^\infty(S^n, \mathbb{R}_+)$ 
such that $\Phi(0)=\widetilde{\gamma}$ and any two $t_1, t_2\in (-\varepsilon, \varepsilon)$, 
there exist $C^\infty$ diffeomorphisms 
$h: S^n\to S^n$ and $H: \mathbb{R}_+\to \mathbb{R}_+$ such that 
the equality $\Phi(t_2)=H\circ \Phi(t_1)\circ h^{-1}$ holds.   
\par 
\smallskip 
The main purpose of this paper is to show the following:   
\begin{theorem}
\label{theorem 1}
The open subset  
$S^\infty(S^n, \mathbb{R}_+)
\cap C^\infty_{\rm conv}(S^n, \mathbb{R}_+)$ 
is dense in $C^\infty_{\rm conv}(S^n, \mathbb{R}_+)$. 
\end{theorem}
\noindent 
Similarly as the assertion (2) of Proposition \ref{proposition 1}, 
Theorem \ref{theorem 1} asserts that any $C^\infty$ convex integrand  
$\gamma: S^n\to \mathbb{R}_+$ can be perturbed to a stable convex integrand $\widetilde{\gamma}$ 
by a sufficiently small perturbation;  
and for any sufficiently small $\varepsilon >0$, any continuous mapping 
$\Phi: (-\varepsilon, \varepsilon)
\to C^\infty_{\rm conv}(S^n, \mathbb{R}_+)$ 
such that $\Phi(0)=\widetilde{\gamma}$ 
and any two $t_1, t_2\in (-\varepsilon, \varepsilon)$, 
there exist $C^\infty$ diffeomorphisms 
$h: S^n\to S^n$ and $H: \mathbb{R}_+\to \mathbb{R}_+$ such that 
the equality $\Phi(t_2)=H\circ \Phi(t_1)\circ h^{-1}$ holds.   
\par 
\bigskip 
In Section \ref{section 2}, preliminaries for the proof of 
Theorem \ref{theorem 1} are given.     
Theorem \ref{theorem 1} is proved in Section \ref{section 3}.  
In Section \ref{section 4}, an application of the proof of 
Theorem \ref{theorem 1} is given.   
%%%%%%%%%%%%%%%%%%%%%%%%%%%%%%%%%%%%%%%%%%%%%%%%%%%%%%%%%%  
%%%%%%%%%%%%%%%%%%%%%%%%%%%%%%%%%%%%%%%%%%%%%%%%%%%%%%%%%% 
\section{Preliminaries}\label{section 2}
Let $\phi:S^n\to\R^{n+1}$ be a $C^\infty$ embedding.   
%, where $M$ is an $n$-dimensional $C^\infty$ manifold and 
Consider the family of functions $F:\R^{n+1}\times S^n\to \R$ defined by
\[
F(v,z)=\frac{1}{2}||\phi(z)-v||^2.
\] 
Notice that $F$ may be regarded as 
a mapping from $\R^{n+1}$ to $C^\infty(S^n,\R)$ 
which maps each $v\in\R^n$ to the function 
$f_v(z)=F(v,z)\in C^\infty(S^n,\R)=\{g:S^n\to \R\;\; C^\infty\}$. 
%For a function $g\in C^\infty(M,\R)$, the notions of non-degenerate 
%critical point and of stability may be defined similarly.   
The set of values $v$ for which $f_v(z)$ has a degenerate critical point, 
denoted by $Caust(\phi)$, is called the \emph{Caustic} of $\phi$  
(for details on caustics,  for instance see 
\cite{arnold, arnold-guseinzade-varchenko, izumiya1, izumiya-takahashi}). 
The set of values $v$ for which $f_v(z)$ 
has a multiple critical value forms the \emph{Symmetry set} of $\phi$, 
denoted by $Sym(\phi)$ 
(for details on symmetry sets, see for instance \cite{bruce-giblin2, 
bruce-giblin, bruce-giblin-gibson}). 
By the assertion (1) of Proposition \ref{proposition 1}, 
these two sets $Caust(\phi)$ and $Sym(\phi)$ 
constitute the set of points $v$ 
for which the function $f_v\in C^\infty(S^n,\R)$ is not stable.
\begin{proposition}\label{caustic}
Let $\phi:S^n\to\R^{n+1}$ be a $C^\infty$ embedding.  
%where $M$ is an $n$-dimensional $C^\infty$ manifold.    
Then, $Caust(\phi)$ has Lebesgue measure zero in $\mathbb{R}^{n+1}$.   
\end{proposition}
\noindent 
For the proof of Proposition \ref{caustic}, see \cite{Morse theory}, \S 6 
\lq\lq Manifolds in Euclidean space\rq\rq.   
\begin{proposition}\label{symmetry set}
Let $\phi:S^n\to\R^{n+1}$ be a $C^\infty$ embedding.  
%where $M$ is an $n$-dimensional $C^\infty$ manifold.    
Then, $Sym(\phi)$ has Lebesgue measure zero in $\mathbb{R}^{n+1}$.
%$Caust(\phi)$ is a subset of Lebesgue measure zero.  
\end{proposition} 
\begin{proof} 
\quad 
Since $\phi$ is an embedding, 
the complement of $\phi(S^n)$ constitutes two connected component{\color{black}s}.   
Denote the bounded connected component by $V_\phi$.   
Set $M=\phi(S^n)$.  For each $\theta\in S^n$, consider the normal vector 
space $N_{\phi(\theta)}(M)$ to $M$ at $\theta$.   
Notice that $N_{\phi(\theta)}(M)$ is a $1$-dimensional vector space.   
Thus, we can uniquely specify the unit vector ${\bf n}(\theta)$ 
of $N_{\phi(\theta)}(M)$ so that $\phi(\theta)+\varepsilon{\bf n}(\theta)$ 
belongs to $V_\phi$ for any sufficiently small $\varepsilon>0$.    
For any $t\in \R$, 
let $\phi_t: S^n\to \mathbb{R}^{n+1}$ be the $C^\infty$ mapping defined by 
$\phi_t(\theta)=\phi(\theta)+t{\bf n}(\theta)$.      
The mapping $\phi_t$ is called a {\it wave front} of $\phi$ 
(for details on wave fronts, see for instance 
\cite{arnold, arnold-guseinzade-varchenko, izumiya1, izumiya-takahashi}).    
It is clear that, by using wave fronts $\{\phi_t\}_{t\in \R}$, 
the set $Sym(\phi)$ can be characterized as follows:   
\[
Sym(\phi)=
\bigcup_{t\in \R}
\left\{
\phi_t(\theta_1)=\phi_t(\theta_2)\; |\; 
\theta_1\ne \theta_2
\right\}.
%\leqno{(*)}
\]
\par 
By Proposition \ref{caustic}, the intersection 
$Sym(\phi)\cap Caust(\phi)$ is of Lebesgue measure zero.   
Thus, in order to show Proposition \ref{symmetry set}, 
it is sufficient to show that $Sym(\phi)\cap (\R^{n+1}-Caust(\phi))$ 
is of Lebesgue measure zero.   
Take one point  
$\phi_{t_0}(\theta_1)=\phi_{t_0}(\theta_2)$ of 
$Sym(\phi)\cap (\R^{n+1}-Caust(\phi))$, where 
$\theta_1, \theta_2$ are two distinct point of $S^n$.        
Set $x_0=\phi_{t_0}(\theta_1)=\phi_{t_0}(\theta_2)$, and 
let $U_0$ be a sufficiently small open neighborhood of $x_0$.    
Notice that, since $Caust(\phi)$ is compact, 
$U_0$ may be chosen so that $U_0\cap Caust(\phi)=\emptyset$.     
%Then, in order to show Proposition \ref{symmetry set}, 
%it is sufficient to show that $Sym(\phi)\cap U_0$ 
%is of Lebesgue measure zero.       
\par 
Let $i$ be $1$ or $2$.    For the $i$, 
define the mapping 
$\left(t_i, \widetilde{\theta}_i\right): U_0\to \R\times S^n$ 
as follows:   
\[
x=\phi_{t_i(x)}(\widetilde{\theta}_i(x)), \quad 
\left(t_i(x_0)=t_0, \widetilde{\theta}_i(x_0)=\theta_i\right).  
\]  
%where $\widetilde{\theta}_i\in S^n$ is sufficiently near $\theta_i$.       
Notice that, since $U_0\cap Caust(\phi)=\emptyset$, both of the following 
two 
are well-defined $C^\infty$ diffeomorphisms.
\begin{eqnarray*}
\left(t_1, \widetilde{\theta}_1\right) & : & 
U_0\to \left(t_1, \widetilde{\theta}_1\right)\left(U_0\right),  \\ 
\left(t_2, \widetilde{\theta}_2\right) & : & 
U_0\to \left(t_2, \widetilde{\theta}_2\right)\left(U_0\right). 
\end{eqnarray*}      
%Notice also that both $t_1, t_2$ are non-singular functions.   
Set $T=t_1-t_2$.    Then, it is clear that 
$Sym(\phi)\cap U_0=T^{-1}(0)$.      
\par 
For any $i=1, 2$, 
let $\nabla t_i(x_0)$ be the gradient vector of $t_i$ at $x_0$.  
Since both $t_1, t_2$ are non-singular functions, 
it follows that neither $\nabla t_1(x_0)$ nor 
$\nabla t_2(x_0)$ 
is the zero vector.    Moreover, from the construction, 
it is easily seen that 
even when $\nabla t_1(x_0)$ and $\nabla t_2(x_0)$ are linearly dependent,   
$\nabla T(x_0)=\nabla t_1(x_0)-\nabla t_2(x_0)$ is a non-zero vector.    
Therefore, taking a smaller open neighborhood 
$\widetilde{U}_0$ of $x_0$ if necessary, it follows that 
$T^{-1}(0)\cap \widetilde{U}_0$ is an $n$-dimensional submanifold of 
$\R^{n+1}$.     %Since $Sym(\phi)$ is a second countable set, 
Therefore, Proposition \ref{symmetry set} follows.   
\end{proof}     
Propositions \ref{caustic}, \ref{symmetry set} clearly yield the following:   
\begin{corollary}\label{corollary 1}
Let $\phi:S^n\to\R^{n+1}$ be a $C^\infty$ embedding.  
%where $M$ is an $n$-dimensional $C^\infty$ manifold.    
Then, the union $Caust(\phi)\cup Sym(\phi)$ is a subset of 
Lebesgue measure zero in $\mathbb{R}^{n+1}$.
\end{corollary} 
%%%%%%%%%%%%%%%%%%%%%%%%%%%%%%%%%%%%%%%%%%%%%%%%%%%%% 
%Let $\gamma,\delta:S^n\to\R_+$ be the $C^\infty$ 
%convex integrands which support $\mathcal{W}_\gamma$ 
%and $\mathcal{DW}_\gamma$, respectively. 
%Notice that $\partial \mathcal{W}_\gamma$ 
%and $\partial \mathcal{DW}_\gamma$ are given 
%by the embeddings $\phi_\delta=inv(graph(\delta))$ 
%and $\phi_\gamma=inv(graph(\gamma))$, respectively.
%\begin{corollary}
%Let $\gamma:S^n\to\R_+$ be a $C^\infty$ convex integrand, 
%then $Caust(\phi_\delta)\cup Sym(\phi_\delta)$ has 
%Lebesgue measure zero.
%\end{corollary}
%\begin{definition}
%Let $\gamma:S^n\to\R_+$ be a $C^\infty$ convex integrand 
%and $\mathcal{W}_\gamma$ be the Wulff shape associated to $\gamma$. %We say that $\mathcal{W}_\gamma$ is {\it of special Morse-type} 
%if the origin does not 
%belong to $Caust(\phi_\delta)\cup Sym(\phi_\delta)$.
%\end{definition}
%%%%%%%%%%%%%%%%%%%%%%%%%%%%%%%%%%%%%%%%%%%%%%%%%%%%%%%%%%%%  
%%%%%%%%%%%%%%%%%%%%%%%%%%%%%%%%%%%%%%%%%%%%%%%%%%%%%%%%%%%% 
\section{Proof of Theorem \ref{theorem 1}}
\label{section 3}
%\begin{theorem}\label{theorem open dense}
%The set $C^\infty_{stable}(S^n,\R_+)\cap C^\infty_{conv}(S^n,\R_+)$ 
%is open and dense in $C^\infty_{conv}(S^n,\R_+)$.
%\end{theorem}
%
%\begin{proof}
Let $\gamma:S^n\to\R_+$ be a $C^\infty$ convex integrand, and 
let $V$ be a neighborhood of $\gamma$ in 
$C^\infty_{\rm conv}(S^n, \R_+)$.      
It is sufficient to show that $V\cap S^\infty(S^n, \R_+)\ne \emptyset$.   
In order to construct an element of  
$V\cap S^\infty(S^n, \R_+)$, we consider the $C^\infty$ 
embedding  
$\phi: S^n\to \mathbb{R}^{n+1}-\{0\}$ defined as follows:    
\[
\phi(\theta)=\left(\theta, \frac{1}{\gamma(-\theta)}\right).
\] 
Let $W$ be the convex hull of $\phi(S^n)$.   
Then, since $\gamma$ is a convex integrand, 
it follows that 
\[
\phi(S^n)=\partial W,
\leqno{(*)}
\] 
where 
$\partial W$ stands for the boundary of $W$.     
%%%%%%%%%%%%%%%%%%%%%%%%%%%%%%%%%%%%%%%%%%%%%%%   
\par 
Next, for any $v\in \mbox{\rm int}(W)$   
consider the parallel translation 
$T_v:\R^{n+1}\to\R^{n+1}$ 
defined by $T_v(x)=x-v$, where 
int$(W)$ means the set consisting of interior points of $W$.   
%, for all $v\in\R^{n+1}$ such that $||v||<
%min\{1,\gamma(\theta);\theta\in S^n\}$.
Moreover, for any $\theta\in S^n$ set $L_\theta=\{(\theta,r)\in\R^{n+1}-\{0\}| r\in\R_+\}$    
and for any $v\in \mbox{\rm int}(W)$ 
define $\tilde{\gamma}_v:S^n\to\R_+$ as follows. 
\[
\left(\theta, \tilde{\gamma}_v(\theta)\right)
=T_v\left(\partial W\right)\cap L_\theta.
\]
Notice that, by $(*)$ and $v\in \mbox{\rm int}(W)$, 
$\widetilde{\gamma}_v$ is a well-defined function.    
%since $W=graph(\tilde{\gamma})(S^n)$,  
Notice also that 
graph$(\tilde{\gamma}_v)=T_v(\partial W)$.     
By $(*)$ and $v\in \mbox{\rm int}(W)$ again, it follows that 
$||\phi(\theta)-v||>0$ for 
any $\theta\in S^n$.     
Thus, it follows that the mapping 
$h_v:S^n\to S^n$ defined by
\[
h_v(\theta)=\dfrac{\phi(\theta)-v}{||\phi(\theta)-v||}
\]
is a $C^\infty$ diffeomorphism and the following holds: 
\[
\left(\tilde{\gamma}_v\circ h_v\right)(\theta)=||\phi(\theta)-v||.
\]
Let $H: \R_+\to \R_+$ be the $C^\infty$ diffeomorphism defined by 
$H(X)=\frac{1}{2}X^2$.   
Then, we have the following:   
%Since $W=graph(\tilde{\gamma})(S^n)$, 
%for all $z\in W$ there exist $\theta\in S^n$ 
%such that $\tilde{\gamma}(\theta)\theta=z$, then
\[
F(v,\theta)=\dfrac{1}{2}||\phi(\theta)-v||^2=
\left(H\circ\tilde{\gamma}_v\circ h_v\right)(\theta).
\]
Hence, we have     
\begin{eqnarray*}
Caust(\phi) & = & 
\left\{v:\exists\theta:\nabla(H\circ\tilde{\gamma}_v\circ h_v)(\theta)=0\,
\,\mbox{\rm and}\,\,\det(Hess(H\circ\tilde{\gamma}_v\circ h_v)(\theta))=0
\right\} \\ 
{ } & = & 
\left\{v:\exists\theta:\nabla(\tilde{\gamma}_v\circ h_v)(\theta)=0\,
\,\mbox{\rm and}\,\,\det(Hess(\tilde{\gamma}_v\circ h_v)(\theta))=0\right\}
\end{eqnarray*}
and
\begin{eqnarray*}
{ } & { } & Sym(\phi) \\ 
{ } & = & 
\left\{v:\exists\theta_1\ne\theta_2:
\nabla(H\circ\tilde{\gamma}_v\circ h_v)(\theta_1)
=\nabla(H\circ\tilde{\gamma}_v\circ h_v)(\theta_2)=0 \right.\\ %\,\,
{ } & { } & \qquad\qquad\qquad\qquad \mbox{\rm and}\,\,\left.(H\circ\tilde{\gamma}_v\circ h_v)
(\theta_1)=(H\circ\tilde{\gamma}_v\circ h_v)(\theta_2)\right\} \\ 
{ } & = &  
\left\{v:\exists\theta_1\ne\theta_2:
\nabla(\tilde{\gamma}_v\circ h_v)(\theta_1)
=\nabla(\tilde{\gamma}_v\circ h_v)(\theta_2)=0\,\,
\mbox{\rm and}\,\,(\tilde{\gamma}_v\circ h_v)
(\theta_1)=(\tilde{\gamma}_v\circ h_v)(\theta_2)\right\}.
\end{eqnarray*}
For any $r\in \mathbb{R}_+$, let 
$B(0, r)$ be the open disk with radius $r$ centered at $0$.   
Then, by Corollary \ref{corollary 1}, 
for any sufficiently small $\varepsilon>0$ 
there exists a point $v\in B(0, \varepsilon)$ such that 
$\widetilde{\gamma}_v\circ h_v$ is stable.  
This implies that there exists a sequence 
$\{v_n\in \mbox{\rm int}(W)\}_{n=1, 2, \ldots}$ 
converging to the origin such that 
$\widetilde{\gamma}_{v_n}$ is stable for any $n\in \mathbb{N}$.  
%for every $V$ open set of the Whitney topology 
%such that $\tilde{\gamma}\in V$, 
%there exist $v\in B(0,\epsilon)
%\setminus (Caust(\phi_\gamma)\cup Sym(\phi_\gamma))$ 
%such that $\tilde{\gamma}_v\circ h\in V$.
\par 
For any $v\in \mbox{\rm int}(W)$, 
define the convex integrand $\gamma_v: S^n\to \mathbb{R}_+$ as follows:  
\[
\gamma_v(\theta)=\frac{1}{\widetilde{\gamma_v}(-\theta)}.  
\]
Since  
%Since $h_0: S^n\to S^n$ is the identity mapping by definition,  
%it follows that 
$S^n$ is compact, the mapping 
$\Phi: B(0, \varepsilon)\to C^\infty_{conv}(S^n,\mathbb{R}_+)$ defined by 
$\Phi(v)=\gamma_v$ is continuous.  
Since $\gamma_0=\gamma$, it follows that if $n$ is sufficiently large, 
then the convex integrand 
$\gamma_{v_n}$ must be inside the given neighborhood $V$ of $\gamma$.    
%Thus, Consequently, given $U$ open set of the Whitney topology 
%such that $\gamma\in U$ 
%there exist $\gamma_v=\dfrac{1}{\tilde{\gamma}_v\circ h}$ 
%such that $\gamma_v\in U$ and $\gamma_v$ is Morse.
%Furthermore, given $G:S^n\times\R_+\to S^n\times\R_+$ 
%defined by $G(\theta,y)=(-\theta,y)$, follows that
%$G\circ graph(\tilde{\gamma}_v\circ h)(S^n)=inv(graph(\gamma_v))(S^n)$, %then $\gamma_v$ is a convex integrand.
%%%%%%%%%%%%%%%%%%%%%%%%%%%%%%%%%%%%%%%%%%%%%%%%%%%%%%%%%%%% 
\hfill $\Box$

%\end{proof}
%%%%%%%%%%%%%%%%%%%%%%%%%%%%%%%%%%%%%%%%%%%%%%%%%%%%%%%%%%%% 
%%%%%%%%%%%%%%%%%%%%%%%%%%%%%%%%%%%%%%%%%%%%%%%%%%%%%%%%%%%%
\section{Application of the proof of Theorem \ref{theorem 1}}
\label{section 4}
%%%%%%%%%%%%%%%%%%%%%%%%%%%%%%%%%%%%%%%%%%%%%%%%%%%%%%%%%%%% 
Theorem \ref{theorem 1} guarantees that 
$C^\infty_{{\rm conv}}(S^n, \R_+)$ has sufficiently many stable functions.   As in \cite{Morse theory}, 
$C^\infty$ convex integrands $\gamma$ having only non-degenerate 
critical points can be a useful tool to investigate 
inv(graph$(\gamma)$).     
For instance, we have the following:  
\begin{proposition}\label{proposition 4}
Let $n$ be an integer satisfying $n\ge 2$.  
Let $\gamma:S^n\to\R_+$ be a $C^\infty$ convex integrand such that 
any critical point is non-degenerate and 
the index of $\gamma$ at any critical point is zero or $n$. 
Then, $\gamma$ must be stable.    
\end{proposition} 
\noindent 
Notice that Proposition \ref{proposition 4} does not hold in the case $n=1$.
\begin{proof}
The following lemma is needed.   
\begin{lemma}\label{index lemma}
Let $n$ be an integer satisfying $n\ge 2$.    
Let $\gamma:S^n\to\R_+$ be a convex integrand having only 
non-degenerate critical points 
such that the index of $\gamma$ at any critical point is zero or $n$. 
Then, $\gamma$ has only two critical points, the index is zero at one point  and it is $n$ at another point.
\end{lemma}
\noindent 
Lemma \ref{index lemma} can be proved by using the Morse inequalities 
as follows.   
For any non-negative integer $\lambda$, 
denote by $C_\lambda$ the number of critical points of index $\lambda$ and by $R_\lambda$ the $\lambda$-th Betti number of $S^n$.
Set 
\[
S_\lambda=R_\lambda-R_{\lambda-1}+\cdots \pm R_0.
\]
Then, the following inequalities, called the {\it Morse inequalities}, hold 
(\cite{Morse theory}):  
\[
S_\lambda\le C_\lambda-C_{\lambda-1}+\cdots \pm C_0.
\leqno{(i_\lambda)}
\]
The inequality $(i_0)$ implies $1\le C_0$.   
Since $n\ge 2$, the inequality $(i_1)$ implies $0-1\le 0-C_0$, 
which is equivalent to 
$C_0\le 1$.     Thus, we have $C_0=1$.    
In the case $\lambda>n$, the inequalities $(i_\lambda), (i_{\lambda+1})$ imply 
the following:   
\[
\sum_{\lambda=0}^n(-1)^\lambda R_\lambda=
\sum_{\lambda=0}^n(-1)^\lambda C_\lambda.    
\]
By using this equality, it is easily seen that $C_n=1$; and thus the proof of  Lemma \ref{index lemma} completes.      
\par 
\medskip 
Now we prove Proposition \ref{proposition 4}.    
It is sufficient to show the following inclusion:   
\[
Sym(\phi)\cap \mbox{\rm int}(W) \subset 
Caust(\phi)\cap \mbox{\rm int}(W).
\]
%Suppose that the origin is 
%a point of $Sym(\phi_\gamma)\setminus Caust(\phi_\gamma)$, 
Here, $\phi: S^n\to \R^{n+1}$ is the $C^\infty$ embedding defined by 
$\phi(\theta)=\left(\theta, \frac{1}{\gamma(-\theta)}\right)$, 
$Sym(\phi)$ is the symmetry set of $\phi$, $Caust(\phi)$ is the caustic of 
$\phi$ and int$(W)$ is the set 
consisting of interior points of the convex hull of inv(graph($\gamma$)).     
Suppose that the following set is non-empty.   
\[
Sym(\phi)\cap \mbox{\rm int}(W) - Caust(\phi)\cap \mbox{\rm int}(W).
\]
Then, take one element $v$ of this set.    
For the $v$, as in Section \ref{section 3}, we construct 
a $C^\infty$ convex integrand $H\circ \gamma_v\circ h_v$.   
Since $v$ is outside $Caust(\phi)\cap \mbox{\rm int}(W)$, 
it follows that all critical point of  $H\circ \gamma_v\circ h_v$ are 
non-degenerate.    
By Lemma \ref{index lemma}, 
the function $H\circ \gamma_v\circ h_v$ 
has only two critical points,  the index is zero at one point 
and it is $n$ at another point.
Let $\theta_1$ (resp., $\theta_2$) be the critical point with index zero 
(resp., index $n$).   
%Thus, the index of $H\circ \gamma_v\circ h_v$ at $\theta_1$ must be zero 
%and the index of $H\circ \gamma_v\circ h_v$ at $\theta_2$ must be $n$.  
Then, it follows that $H\circ \gamma_v\circ h_v(\theta_1)$ is the minimal value and $H\circ \gamma_v\circ h_v(\theta_2)$ is the maximal value. 
On the other hand, 
since $v$ is inside $Sym(\phi)\cap \mbox{\rm int}(W)$, 
there exist two distinct critical points $\widetilde{\theta}_1, 
\widetilde{\theta}_2\in S^n$ 
such that 
$H\circ \gamma_v\circ h_v(\widetilde{\theta}_1)
=H\circ \gamma_v\circ h_v(\widetilde{\theta}_2)$.            
Since there are no critical points of $H\circ \gamma_v\circ h_v$ except for 
$\theta_1, \theta_2$, it follows that 
$\{\theta_1, \theta_2\}=\{\widetilde{\theta}_1, \widetilde{\theta}_2\}$.    
Therefore, $H\circ \gamma_v\circ h_v$ must be a constant function, and 
this implies that $\phi(S^n)$ is a sphere centered at $v$.  
Hence, it follows that $Caust(\phi)=Sym(\phi)=\{v\}$, which contradicts 
the assumption that $v\not\in Caust(\phi)$.    
Therefore, the following inclusion holds:  
\[
Sym(\phi)\cap \mbox{\rm int}(W) \subset 
Caust(\phi)\cap \mbox{\rm int}(W).
\]    
\end{proof}    
%%%%%%%%%%%%%%%%%%%%%%%%%%%%%%%%%%%%%%%%%%%%%%%%%%%%%%%%%%%% 
%%%%%%%%%%%%%%%%%%%%%%%%%%%%%%%%%%%%%%%%%%%%%%%%%%%%%%%%%%%%  
%%%%%%%%%%%%%%%%%%%%%%%%%%%%%%%%%%%%%%%%%%%%%%%%%%%%%%%%%%%% 
%%%%%%%%%%%%%%%%%%%%%%%%%%%%%%%%%%%%%%%%%%%%%%%%%%%%%%%%%%%% 

\end{document}